\documentclass[leqno,a4paper]{article}
\usepackage{amsmath,bbm, amsthm,amssymb}
\usepackage{eufrak}

\newcommand{\mK}{\mathcal{K}}

\newcommand{\be}{\boldsymbol{e}}
\newcommand{\bu}{\boldsymbol{u}}

\newcommand{\ff}{\boldsymbol{f}}

\newcommand{\bC}{\boldsymbol{C}}
\newcommand{\bD}{\boldsymbol{D}}
\newcommand{\bL}{\boldsymbol{L}}
\newcommand{\bS}{\boldsymbol{S}}
\newcommand{\bT}{\boldsymbol{T}}

\newcommand{\R}{{\mathbb R}}

\newcommand{\oom}{{\overline{\omega}}}

\newcommand{\na}{{\nabla}}
\newcommand{\pa}{{\partial}}
\newcommand{\al}{{\alpha}}
\newcommand{\te}{{\theta}}
\newcommand{\bt}{{\beta}}

\newcommand{\ro}{{\rho}}

\newcommand{\la}{{\lambda}}
\newcommand{\de}{{\delta}}
\newcommand{\om}{{\omega}}
\newcommand{\sg}{{\sigma}}
\newcommand{\z}{{\zeta}}

\newcommand{\Om}{{\Omega}}
\newcommand{\De}{{\Delta}}
\newcommand{\Ov}{{\overline{\Omega}}}
\newcommand{\Ga}{{\Gamma}}

\newcommand{\ep}{{\epsilon}}

\newcommand{\ph}{{\phi}}

\newcommand{\ho}{{\widehat{\omega}}}

\newcommand{\bB}{\boldsymbol{B}}

\def\vs1{\vspace{1ex}}
\newtheorem{definition}{Definition}[section]
\newtheorem{theorem}{Theorem}[section]
\newtheorem{lemma}[theorem]{Lemma}
\newtheorem{proposition}{Proposition}[section]

\numberwithin{equation}{section}
\def\be{\begin{equation}}
\def\ee{\end{equation}}
\def\ed{\end{document}}
\begin{document}
\title{\bf\normalsize On classical solutions to
elliptic boundary value problems. The full regularity spaces $\,
C^{0,\,\la}_\al(\Ov) \,$. }
\author{by H.~Beir\~ao da Veiga}
%\date{}
\maketitle
\begin{abstract}
Let $\,\bL\,$ be a second order, uniformly elliptic operator, and
consider the equation $-\,\bL\, u=\,f \,$ under the homogeneous
boundary condition $\, u=\,0 \,.$ It is well known that $\,f \in
C(\Ov)\,$ does not guarantee $\,\na^2\,u \in C(\Ov)\,$. This gap led
to look for functional spaces $\,C_*(\Ov)\subset\,C(\Ov)\,,$ as
large as possible, for which $\,f\in \,C_*(\Ov)\,$ \emph{merely}
guarantees the continuity of $\,\na^2\,u\,$ (but nothing more, say).
H\"older continuity is too restrictive to fulfill this minimal
requirement since in this case $\,\na^2\,u\,$ inherits the whole
regularity enjoyed by $\,f\,$ (we say that \emph{full regularity}
occurs). This two opposite situations led us to look for significant
cases in which \emph{intermediate regularity} (i.e., between
\emph{mere continuity} and \emph{full regularity}) occurs. This
holds for data in Log spaces $\, D^{0,\,\al}(\Ov)\,,$
$\,0<\,\al<\,+\infty\,,$ simply obtained  by replacing in the
modulus of continuity of H\"older spaces the quantity $
\,1/|\,x-\,y|\,$ by $\, \log\,(\,1/|\,x-\,y|)\,.$ If $\,f \in
D^{0,\,\al},$ for some fixed $\,\al>\,1\,,$ then $\,\na^2\,u \in
D^{0,\,\al-\,1}\,.$ This regularity is optimal.\par%
The above picture opened the way to further investigation. Below we
study the more general problem of data $\,f\,$ in subspaces of
continuous functions $\,D_{\oom}\,$, characterized by a given
\emph{modulus of continuity} $\,\oom(r)\,.$ H\"older and Log spaces
are particular cases. A significant new, lets say curious, case is
shown by the family of functional spaces $\, C^{0,\,\la}_\al(\Ov)
\,,$ $\,0 \leq\,\la<\,1\,$, $\,\al \in\,\R\,$. In particular,
$\,C^{0,\,\la}_0(\Ov)=\,C^{0,\,\la}(\Ov)\,$, and
$\,C^{0,\,0}_\al(\Ov)=\, D^{0,\,\al}(\Ov)\,$. Main point is that
full regularity occurs for $\,\la>\,0\,$ and arbitrary $\,\al
\in\,\R\,$. If $\,f \in\, C^{0,\,\la}_\al(\Ov) \,$ then $\,\na^2\,u
\in C^{0,\,\la}_\al(\Ov)\,$.

 \vspace{.2cm}


\vspace{.2cm}

\end{abstract}
\bibliographystyle{amsplain}
\section{Introduction.}\label{introduction}
We start by some notation. By $\Om$ we denote an open, bounded,
connected set in $\R^n\,$, locally situated on one side of its
boundary $\,\Ga\,.$ The boundary $\,\Ga\,$ is of class
$\,C^{2,\,\la}\,,$ for some $\,\la\,,$ $\,0<\,\la \leq \,1\,.$
Notation $\Om_0 \subset \subset \Om$ means that the open set
$\Om_0$ satisfies the property $\Ov_0 \subset \Om$.\par%
By $\,C(\Ov)\,$ we denote the Banach space of all real continuous
functions $\,f\,$ defined in $\,\Ov\,$. The "sup" norm is denoted by
$ \|\,f\,\|\,. $ We also appeal to the classical spaces
$\,C^k(\Ov)\,$ endowed with their usual norms $ \|\,u\,\|_k\,,$ and
to the H\"older spaces $\,C^{0,\,\la}(\Ov)\,,$ endowed with the
standard semi-norms and norms. The space $\,C^{0,\,1}(\Ov)\,,$ is
sometimes denoted by $\,Lip\,(\Ov)\,,$ the space of Lipschitz
continuous functions in $\,\Ov\,.$ We set
$$
I(x;\,r)=\,\{\,y:\,|y-\,x| \leq\, r\,\}\,, \quad\,\Om(x;\,r)=\, \Om
\,\cap\,I(x;\,r)\,.
$$
Symbols $c\,$ and $\,C\,$ denote generical positive
constants. We may use the same symbol to denote different constants.%

\vspace{0.2cm}

Let us present some reasons that led us to the present study. We say
that solutions to a specific boundary value problem are
\emph{classical} if all derivatives appearing in the equations and
boundary conditions are continuous up to the boundary on their
domain of definition. We call \emph{"minimal assumptions problem"}
the investigation of "minimal assumptions" on the data which
guarantee that solutions are classical. The very starting point of
these notes was reference \cite{BVJDE}, where the main goal was to
look for \emph{minimal assumptions} on the data which guarantee
classical solutions to the $\,2-D\,$ Euler equations in a bounded
domain. The study of this problem led to the auxiliary problem
\begin{equation}
\left\{
\begin{array}{l}
\bL\,u=\,f \quad \textrm{in} \quad \Om \,,\\
u=\,0 \quad \textrm{on} \quad \Ga\,.
\end{array}
\right.
\label{lapnao}
\end{equation}
We do not discuss here the relation between the Euler equations and
problem \eqref{lapnao}. The interested reader is referred to the
original paper \cite{BVJDE}, and also to \cite{BVJP}, where a
complete description is presented.\par%
Below we consider second order, uniformly elliptic operators
\begin{equation}
\bL=\,\sum_1^n a_{i\,j}(x)\, \pa_i\,\pa_j\,.%
\label{elle}
\end{equation}
Without loss of generality, we assume that the matrix of
coefficients is symmetric. To avoid conditions depending on the
single case, we assume once and for all that the operator's
coefficients are Lipschitz continuous in $\,\Ov\,.$
Lower order terms can be considered without difficulty.\par%
A H\"older continuity assumption on $\,f\,$ is unnecessarily
restrictive to guarantee $\,\na^2\,u\in\,C(\Ov)\,,$ where $\,u\,$ is
the solution to problem \eqref{lapnao}. On the other hand,
continuity of $\,f\,$ is not sufficient to guarantee continuity of
$\,\na^2\,u\,.$ This situation led us to consider in \cite{BVJDE} a
Banach space $\,\,C_*(\Ov)\,$, $\, C^{0,\,\la}(\Ov)\subset
\,C_*(\Ov)\subset\,C(\Ov)\,,$ for which the following result holds
(Theorem 4.5, in \cite{BVJDE}).
\begin{theorem}
Let $\,f \in \,C_*(\Ov)\,$ and let $\,u\,$ be the solution of
problem \eqref{lapnao}. Then $\,u \in\, C^2(\Ov)\,,$ moreover,
\begin{equation}
\|\,\na^2\,u\,\| \leq \,c\,\|\,f\,\|_*\,.%
\label{lapili}
\end{equation}
\label{laplaces}
\end{theorem}
The above result was stated for constant coefficients operators,
however the proof applies without any modification to variable
coefficients case, since it is based on some properties of the Green
functions, which hold in the general case.

\vspace{0.2cm}

For the readers convenience we recall definition and main properties
of $\,C_*(\Ov)\,$ (see \cite{BVJDE} and, for complete proofs,
\cite{BVSTOKES}).
Define, for $\,f \in \,C(\Ov)\,,$ and for each $\,r>\,0\,$,%
\begin{equation}
\om_f(r) \equiv \, \sup_{\,x,\,y
\in\,\Om \,;\, 0<\,|x-\,y| \leq\,r } \,|\,f(x)-\,f(y)\,|\,,%
\label{cinco}
\end{equation}
and consider the semi-norm
\begin{equation}%
[\,f\,]_* =\,[\,f\,]_{*,\,R} \equiv \int_0^R \,\om_f(r) \,\frac{dr}{r}\,,%
\label{seis}
\end{equation}
where $\,R>\,0\,$ is fixed. The finiteness of the above integral is
known as \emph{Dini's continuity condition}. We define the functional space
\begin{equation}
C_*(\Ov) \equiv\,\{\,f \in\,C(\,\Ov): \,[\,f\,]_*
<\,\infty\,\}%
\label{cstar}
\end{equation}
normalized by $ \,\|\,f\,\|_* =\,[\,f\,]_*+\,\|\,f\,\|\,.$ Norms
defined for two distinct values of $\,R\,$ are equivalent. We have
shown that $\,C_*(\Ov)\,$ is a Banach space, that the embedding $\,
C_*(\Ov) \subset \,C(\Ov)\,$ is compact, and that
the set $\,C^{\infty}(\Ov)\,$ is dense in $\,C_*(\Ov)\,.$%
\vspace{0.2cm}

The regularity theorem \ref{laplaces} for data in $\,C_*(\Ov)\,$
raise a number of new questions. Contrary to the case of H\"older
continuity, where full regularity is restored ($\,\na^2\,u\,$ and
$\,f\,$ has the same regularity), no significant additional
regularity is obtained for data in $\,C_*(\Ov)\,$, besides mere
continuity of $\,\na^2\,u\,.$ So, we are in the presence of two
totally opposite behaviors. An "intermediate" situation is shown by
the Log spaces $\, D^{0,\,\al}(\Ov)\,,$ $\,0<\,\al<\,+\infty\,.$ In
the constant coefficients case, if $\,f \in
D^{0,\,\al}_{loc}(\Om)\,$ for fixed $\,\al>\,1\,,$  then $\,\na^2\,u
\in D^{0,\,\al-\,1}_{loc}(\Om)\,.$ This regularity result is
\emph{optimal}. Furthermore, it holds up to "flat boundary points".
See theorem \ref{laplohas} below.%

\vspace{0.2cm}

The above picture leads us to consider general data spaces $\,D_\oom
(\Ov)\,$, characterized by a given \emph{modulus of continuity}
function $\,\oom(r)\,.$ These spaces are contained between
$\,Lip(\Ov)\,$ and $\,C_*(\Ov)\,$. H\"older and Log spaces are
particular cases. To each suitable $\,\oom(r)\,$ there corresponds a
$\,\ho(r)\,$ such that $\,\na^2\,u \in D_\ho\,$ for $\,f \in
\,D_\oom\,,$ see theorem \ref{sufasvero}. Clearly, $\,\oom(r)\leq
\,c\,\ho(r)\,$, for some $c>0$. This situation occurs for data in
Log spaces, see theorem \ref{laplolas}. Furthermore, if a reverse
inequality $\,\ho(r)\leq \,c\,\oom(r)\,$ holds, then full regularity
occurs, see theorem \ref{sufasvero-2}. This is the situation for
data in H\"older spaces. A more general, quite significant, case of
full regularity concerns the new family of functional spaces
$\,C^{0,\,\la}_\al(\Ov)\,,$ $\,0 \leq\,\la<\,1\,$, $\,\al
\in\,\R\,$, called here H\"olog spaces. For $\,\la>\,0\,$ and
$\,\al=\,0\,,$ $\,C^{0,\,\la}_0(\Ov)=\,C^{0,\,\la}(\Ov)\,,$ is a
H\"older classical space. For $\,\la=\,0\,$ and $\,\al>0\,$,
$\,C^{0,\,0}_\al(\Ov)=\, D^{0,\,\al}(\Ov)\,$ is a Log space. Main
point is that, for $\,\la>\,0\,,$ $\,\na^2\,u $ and $\,f\,$ enjoy
the same $\,C^{0,\,\la}_\al(\Ov)\,$ regularity (full regularity).
See theorem \ref{laplohas}.\par%
The assumptions on the data spaces $\,D_\oom (\Ov)\,$ required in
theorems \ref{sufasvero} and \ref{sufasvero-2} can be substantially
weakened. However, explicit statements in this direction would not
add particularly significant features, at the cost of more involved
manipulations.%
\section{The spaces $\,D_\oom (\Ov)\,$. General properties.}\label{novicas}
In this section we define the spaces $\,D_\oom (\Ov)\,$ and state
some general properties. We consider real, \emph{continuous},
\emph{non-decreasing} functions $\,\oom(r)\,$, defined for
$\,0\,\leq r\,<\,R\,,$ for some $\,R>\,0\,.$ Furthermore,
$\,\oom(0)=\,0\,,$ and $\,\oom(0)>0\,$ for $\,r>\,0\,.$ These three
conditions are assumed everywhere in the sequel. Sometimes, the
functions $\,\oom(r)\,$ will be called \emph{oscillation functions}.\par%
Recalling \eqref{cinco}, we set
\begin{equation}
[f]_\oom = \, \sup_{\,0<\ r\,<\,R } \,\frac{\om_f(r)}{\oom(r)}\,.%
\label{fom}
\end{equation}
Hence,
\begin{equation}
\om_f(r) \leq\, [f]_\oom  \,\oom(r)\,, \quad \forall \, r \in(0,\,R)\,. %
\label{fom2}
\end{equation}
Further, we define the linear space
\begin{equation}
D_\oom(\Ov) =\,\{\, f\in\,C(\Ov) :\, [f]_\oom  <\,\infty\,\}\,.
\label{defdom}
\end{equation}
One easily shows that $\,[f]_\oom\,$ is a semi-norm in
$\,D_\oom(\Ov)\,.$ We introduce a norm by setting
\begin{equation}
\|f\|_\oom =\,[f]_\oom +\, \|f\|\,.%
\label{omnorm}
\end{equation}
Two norms with distinct values of the parameter $\,R\,$ are
equivalent, due to the addition of $\,\|f\|\,$ to the semi-norms.\par%
It is worth noting that, beyond the three conditions on
$\,\oom(r)\,$ introduced above, any other property assumed in the
sequel is merely needed in an arbitrarily small neighborhood of the
origin. This fact may be used without a continual reference. In the
sequel, to avoid continual specification, we introduce the following
definitions.
\begin{definition}
We say that $\,\oom(r)\,$ is \emph{concave} if it is concave in a
neighborhood of the origin, and say that $\,\oom(r)\,$ is
\emph{differentiable} if it is point-wisely differentiable (not
necessarily continuously differentiable), for each
$\,r>\,0\,,$ in a neighborhood of the origin.%
\label{finacas}
\end{definition}
Next we establish some useful properties of the above functional
spaces.
\begin{proposition}
If
\begin{equation}
0<\,k_0 \leq\,\frac{\oom(r)}{\oom_0(r)} \leq\,k_1<\,+\infty\,,
\label{pertinho}
\end{equation}
for $\,r\,$ in some neighborhood of the origin, then
$\,D_\oom(\Ov)=\,D_{\oom_0}(\Ov)\,,$ with \emph{equivalent norms}.\par%
\label{eqnormes}
\end{proposition}
The proof is immediate.
\begin{lemma}
If $\,\|f_n\|_\oom \leq\,C_0\,,$ and $\,f_n \rightarrow \,f\,$ in
$\,C(\Ov)\,$ then $\,\|f\|_\oom \leq\,C_0\,.$%
\label{umlemmas}
\end{lemma}
The proof is immediate.
\begin{theorem}
$D_\oom(\Ov)$ is a Banach space.%
\label{completos}
\end{theorem}
\begin{proof}
Let $f_n$ be a Cauchy sequence in $D_\oom(\Ov)\,.$ It follows, in
particular, that $\,f_n \rightarrow \,f\,$ in $\,C(\Ov)\,,$ where
$\, f \in \,D_\oom(\Ov)\,.$ On the other hand, for $\,|x-\,y|=\,r\,$,
$$
\begin{array}{l}
\frac{|\,(f(x)-\,f_n(x)\,)
-(f(y)-\,f_n(y)\,)\,}{\oom(r)}=\\%
\\
\lim_{m \rightarrow \,\infty}\, \frac{|\,(f_m(x)-\,f_n(x)\,)
-(f_m(y)-\,f_n(y)\,)\,}{\oom(r)}\leq\,\limsup_{m \rightarrow
\,\infty} \,[\,f_m-\,f_n\,]_\oom\,.%
\end{array}
$$
Hence
$$
[\,f-\,f_n\,]_\oom \leq\,\limsup_{m \rightarrow
\,\infty} \,[\,f_m-\,f_n\,]_\oom\,.%
$$
From the Cauchy sequence hypothesis it readily follows that
$$
\lim_{n \rightarrow
\,\infty} \,[\,f-\,f_n\,]_\oom=\,0\,.%
$$
\end{proof}
Next we consider compact embedding properties. In the sequel, $ \oom
\,<<\, \oom_1 $ mean that
\begin{equation}
\lim_{r\rightarrow\,0}\,\frac{\oom(r)}{\oom_1(r)}=\,0\,.%
\label{pactos}
\end{equation}
\begin{theorem}
Assume that $ \oom << \oom_1\,.$ Then the embedding
$$
D_\oom(\Ov)\,\subset\,D_{\oom_1}(\Ov)\,,
$$%
is compact.%
\label{compactos}
\end{theorem}
\begin{proof}
By assumption
$$
\|\,f_n\,\|_\oom =\,[\,f_n\,]_\oom +\,\|\,f_n\,\| \leq\,C_0, \quad \forall \, n\,.
$$
From \eqref{pactos} it follows that $\oom(r) \leq \,\oom_1(r)\,$ for
$\,r\in\,(\,0,\,R_0)\,,$ for some $\,R_0>\,0\,$.  For $\,r \in (R_0,\,R)$  one has
$\,\oom(r)\leq\frac{\oom(R)}{\oom_1(R_0)}\,\oom_1(r)\,.$ So there is a
positive constant $\,C\,$ such that
$$
\oom(r)\leq \,C\,\oom_1(r)\,,\quad \forall \,r\in\,(0,\,R)\,.
$$
By the Ascoli-Arzela Theorem, the embedding%
$$
D_\oom(\Ov)\subset\,C(\Ov)
$$
is compact. Hence, by appealing to lemma \ref{umlemmas}, one shows that there is a
subsequence, still denoted $\,f_n\,$, which converges uniformly to
some $\,f\in\,D_\oom(\Ov)\,.$ Without loss of generality, we
assume that $\,f=\,0\,.$\par%
Let $\,|\,x-\,y\,|=\,r\,.$ One has
$$
\frac{|\,f_n(x)-\,f_n(y)\,|}{\oom_1(r)}=\,\frac{|\,f_n(x)-\,f_n(y)\,|}{\oom(r)}\,\frac{\oom(r)}{\oom_1(r)}\,,
 \quad \forall \, n\,.
$$
Given $\,\ep>\,0\,,$ it follows from \eqref{pactos} that there is
$\,R_0(\ep) >\,0\,$ such that
\begin{equation}
0<\,r\leq\,R_0(\ep)\, \implies   \frac{\oom(r)}{\oom_1(r)}<\,\ep\,. %
\label{evon}
\end{equation}
Hence, for $\,0<\,|x-\,y|\leq\,R_0(\ep)\,,$
\begin{equation}
\frac{|\,f_n(x)-\,f_n(y)\,|}{\oom_1(r)} \leq\,C_0\,\ep\,,\quad  \forall \,n\,.%
\label{zed}
\end{equation}
 On the other hand, if
$\,r \in (\,R_0(\ep),\,R)\,,$ one has
$$
\frac{|\,f_n(x)-\,f_n(y)\,|}{\oom_1(r)} \leq
\,\frac{2}{\oom_1(R_0(\ep))}\,\|f_n\|\,.
$$
Since the sequence $\,\|f_n\|\,$ converges to zero, there is an index $\,N(\ep)$
such that, for each $\,n>\,N(\ep)\,,$ the right hand side of the last inequality is smaller than
$\ep\,.$ This fact, together with
\eqref{zed}, shows that \eqref{zed}  holds for $\,0<\,|x-\,y|\leq\,R\,$ and
 $\,n>\,N(\ep)\,$ (increase the constant $\,C_0\,,$  if necessary). So,
$$
\lim_{n\rightarrow\,+\,\infty} \,[\,f_n\,]_{\oom}=\,0\,.
$$
\end{proof}
\begin{lemma}
Assume that $\,\oom\,$ is concave. Then
\begin{equation}
\oom(k\,r) \leq\,k\,\oom(r)\,, \quad \forall \,k\geq\,1\,.%
\label{concavo}
\end{equation}
\label{cocas}
\end{lemma}
The proof is immediate.\par%
In reference \cite{BV-arxiv}, Theorem 4.4, we claimed that
$\,C^{\infty}(\Ov)\,$ is dense in Log spaces, leaving the proof to
the reader. This result is wrong, as shown below. It is worth noting
that $\,C^{\infty}(\Ov)\,$ is dense in $\,C_*(\Ov)\,,$ a result that
has a central role in reference \cite{BVSTOKES}.
\begin{theorem}
Assume that $\,\oom(r)\,$ is concave and that
$\,\oom_1(r)<<\,\oom(r)\,.$ Then $\,D_{\oom_1}(\Ov)\,$ is not dense
in $\, D_\oom(\Ov)\,.$%
\label{osned}
\end{theorem}
\begin{proof}
We assume that the origin belongs to $\,\Om\,,$ and argue in a
neighborhood $\,I=\,I(0,\,\de) \subset \Om\,.$ Define $\,f\,$ by
setting $\,f(x)=\,\oom(|x|)\,.$ We show that
$\,[\,f-\,g]_{\,\oom}\geq\,1\,,$ for each $\,g \in
D_{\oom_1}(\Ov)\,.$ It is sufficient to consider the one-dimensional
case. One has
$$
\frac{|\,(f(x)-\,g(x)\,)-\,(\,f(0)-\,g(0)\,)\,|}{\oom(|x|)}=\,\Big|\,1-\,\frac{\,g(x)-\,g(0)\,}{\oom(|x|)}\,\Big|\,.
$$
Hence $\,[\,f-\,g]_{\,\oom}\geq\,1\,$ follows, if we show that
\begin{equation}
\lim_{x\rightarrow\,0} \,\frac{\,g(x)-\,g(0)\,}{\oom(|x|)}=\,0\,.
\label{desejada}
\end{equation}
 Let's prove this last inequality. One has, as $\,x
\rightarrow\,0\,,$
\begin{equation}
\lim\,\frac{\,g(x)-\,g(0)\,}{\oom(|x|)}=\,\lim\,\frac{\,g(x)-\,g(0)\,}{\oom_1(|x|)}
\cdot\,\lim\,\frac{\oom_1(|x|)}{\oom(|x|)}=\,0\,.%
\label{liminfas}
\end{equation}
\end{proof}

\vspace{0.2cm}

Note that in the above proof we did not explicitly appeal to the
concavity assumption. This assumption was introduced merely to
guarantee that $\,f(x)=\,\oom(|x|)\,$ belongs to $\, D_\oom\,$ in a
neighborhood of the origin. This holds if
\begin{equation}
\oom(s)\leq\,\oom(r)+\,c\,\oom(s-r)\,, \quad \textrm{for} \quad
\,0<\,r<\,s<\,\ro\,,%
\label{regdens}
\end{equation}
for some constant $c\geq\,1\,,$ and some $\,\ro>\,0\,.$
By lemma \ref{cocas}, concave oscillation functions  satisfy \eqref{regdens} with $c=\,1\,.$\par%
The above result shows, in particular, that $\,C^{0,\,\mu}(\Ov)\,$
is not dense in  $\,C^{0,\,\la}(\Ov)\,$ for $1\geq\,\mu> \la
>\,0\,.$ In particular $Lip\,(\Ov)\,$, hence $\,C^1(\Ov)\,$, is not
dense in  $\,C^{0,\,\la}(\Ov)\,$.

\vspace{0.2cm}

We end this section by stating an extension theorem, where $
\,\Om_{\de}\equiv\,\{\,x:\, dist(x,\,\Om\,) <\,\de\,\}\,.$
\begin{theorem}
Assume that $\,\Om\,$ is convex or, alternatively, that
$\,\oom(r)\,$ is concave (concavity may be replaced by  condition
\eqref{comka}) . Then there is a $\,\de>0\,$ such that the
following holds. There is a linear continuous map $\,T\,$ from
$\,C(\Ov)\,$ to $\,C(\Ov_{\de})\,,$ and from $\,D_\oom(\Ov)\,$ to
$\,D_\oom(\Ov_{\de})\,,$ such that
$\,T\,f(x)=\,f(x)\,$, for each $\,x \in\,\Ov\,.$%
\label{bahois}
\end{theorem}
The proof follows by appealing to the argument used to prove the
Theorem 2.3 in \cite{BVSTOKES}. See reference \cite{BV-arxiv}. Note
that the classical proof of approximation of functions on compact
subsets of $\,\Om\,$ by appealing to mollification, does not work
here. Otherwise, the density property refused by theorem \ref{osned}
would hold.

\section{Spaces $\,D_{\ho}(\Ov)\,$ and regularity. The main theorems.}\label{domega}
In this section we state the theorems \ref{sufasvero} and
\ref{sufasvero-2}. From now on we assume that the modulus of
continuity $\,\oom(r)\,$ satisfy the condition
\begin{equation}
\int_0^R \,\oom(r) \,\frac{dr}{r}\,\leq\,C_R\,,%
\label{massim}
\end{equation}
for some constant $C_R\,.$  Assumption \eqref{massim} is equivalent
to the inclusion $\, D_\oom(\Ov)\subset\,C_*(\Ov)\,.$ This assumption is almost
necessary to obtain $\,\na^2\,u \in\, C(\Ov)\,.$\par%
We put each suitable oscillation function $\,\oom(r)\,$ in
correspondence with a unique, related oscillation function
$\,\ho(r)\,$ defined by setting $\ho(\,0)=\,0\,,$ and
\begin{equation}
\ho(\,r)=\,\int_0^r \,\oom(s) \,\frac{ds}{s}%
\label{chapeu}
\end{equation}
for $\,0<\,r\leq\,R\,.$ Hence, to a functional space
$\,D_\oom(\Ov)\,$ there corresponds a well defined functional space
$\,D_{\ho}(\Ov)\,.$ Obviously, $\,\ho\,$ satisfies all the
properties described in section \ref{novicas} for generical
oscillation functions. In particular, Banach spaces
\begin{equation}
D_\ho(\Ov) =\,\{\, f\in\,C(\Ov) :\, [f]_\ho <\,\infty\,\}
\label{defdom2}
\end{equation}
turn out to be well defined.\par%
Next we discuss some additional restrictions on the data spaces
$D_\oom(\Ov)$. We start by excluding $\,Lip(\Ov)\,$ as data space
since this \emph{singular} case, largely considered in literature,
is borderline. In fact, to assign $\,f \in\,Lip\,(\Ov)\,$ is
equivalent to assign $\,\na\,f\in\,L^\infty(\Om)\,,$ which is the
starting point of a new chapter. So, we impose the \emph{strict}
limitation
\begin{equation}
Lip(\Ov)\, \subset\, D_\oom(\Ov)\, \subset\,C_*(\Ov)\,.%
\label{namely}
\end{equation}
 Exclusion of $\,Lip(\Ov)\,$ means that $\,\oom(r)\,$ does not verify
$ \,\oom(r) \leq\,c\,r\,,$ for any positive constant $\,c\,.$ Hence
$\,\limsup (\,\oom(r)/\,r\,)=\,+\,\infty\,,$ as $\,r
\rightarrow\,0\,$. We simplify, by assuming that
\begin{equation}
\lim_{r \rightarrow\,0}\, \frac{\oom(r)}{r}=\,+\,\infty\,.%
\label{simplelas}
\end{equation}
In particular, the graph of $\,\oom(r)\,$ is tangent to the vertical
axis at the origin (as for H\"older and Log spaces). This picture
also shows that \emph{concavity} of the graph is here a quite
natural assumption. Concavity implies that left and right
derivatives are well defined, for $r>\,0\,$. By also taking into
account that $\,\oom(r)\,$ is non-decreasing, we realize that
pointwise differentiability of $\,\oom(r)\,,$ for $r>\,0\,$, is not
a particularly restrictive assumption. This last claim is reenforced
by the equivalence result for norms, under condition
\eqref{pertinho}. This equivalence allows regularization of
oscillation functions $\,\oom(r)\,,$ staying inside the same
original functional space $\,D_\oom(\Ov)\,.$ Summarizing,
\emph{differentiability} and \emph{concavity} (recall definition
\ref{finacas}) are natural assumptions
here.\par%
If $\,\oom(r)\,$ is concave, not flat, and differentiable, it follows that
\begin{equation}
\frac{\oom(r)}{r\,\,\oom'(r)}> \,1\,,%
\label{flitas}
\end{equation}%
for all $r>\,0\,.$ This justifies the assumption
%\begin{equation}
%\frac{\oom(r)}{r\,\,\oom'(r)} \geq\,C_1 > \,1\,,%
%\label{flitas2}
%\end{equation}%
%in a neighborhood of the origin, or even
\begin{equation}
\lim_{r\rightarrow 0}\,\frac{\oom(r)}{r\,\,\oom'(r)}=\,C_1 > \,1\,,%
\label{naflitas}
\end{equation}%
where $\,C_1=\,+\,\infty\,$ is admissible. Assumption \eqref{naflitas} is reenforced by the particular
situation in Lipschitz, H\"older, and Log cases. The
limit exists and is given by, respectively, $\,1\,$,
$\,\frac{1}{\la}\,,$ and $\,+\,\infty\,.$ As expected, the Lipschitz
case stays outside the admissible range. Note that, basically, the
larger is the space, the larger is the limit.\par%
The above consideration allow us to assume in theorems
\ref{sufasvero} and \ref{sufasvero-2} that oscillation functions
$\,\oom(r)\,$, are concave, differentiable, and satisfies conditions
\eqref{massim}, \eqref{simplelas}, and \eqref{naflitas}.\par%
Note that, due to a possible loss of regularity, it could happen
that a $\,D_{\ho}(\Ov)\,$ space  is not contained in $\,C_*(\Ov)\,,$
as happens in theorem \ref{laplolas}, if $\,1<\al<2\,.$ In other
words, $\,\ho(r)\,$ does not necessarily satisfy \eqref{massim}.\par
Next, we define the quantity
\begin{equation}
B(r)=:\frac{\,r\,\int_r^{R} \,\frac{\oom(s)}{s^2} \,ds}{\int_0^r
\,\frac{\oom(s)}{s}\,ds}\,.%
\label{eeles}
\end{equation}
\label{assasdois}%
The following result holds.
\begin{lemma}
Assume that $\,\oom(r)\,$ is concave and satisfies assumptions
\eqref{massim}, \eqref{simplelas} and \eqref{naflitas}. Then
\begin{equation}
\lim_{r\,\rightarrow \,0}\,B(r)=\,\frac{1}{C_1-\,1}\,.%
\label{peles}
\end{equation}
In particular there is a positive constant $\,C_2\,$ such that
\begin{equation}
B(r)\leq\,C_2\,%
\label{eeles-2}
\end{equation}
in some neighborhood of the origin.%
\label{beat}
\end{lemma}
\begin{proof}
By appealing to  \eqref{massim},  \eqref{simplelas} and to a de
L'H\^opital's rule one shows that
\begin{equation}
\lim_{r\,\rightarrow \,0} \,\frac1r \,\int_0^{r} \,\frac{\oom(s)}{s}
\,ds =\,+\,\infty\,.%
\label{lecas}
\end{equation}
On the other hand
\begin{equation}
\lim_{r\,\rightarrow \,0}\,B(r)= \,\lim_{r\,\rightarrow \,0}
\frac{\int_r^{R} \,\frac{\oom(s)}{s^2} \,ds}{\frac1r \,\int_0^r
\,\frac{\oom(s)}{s}\,ds}\,.%
\label{doente}
\end{equation}
Equation \eqref{lecas} shows that the denominator $\,g(r)\,$ of the
fraction in the right hand side of \eqref{doente} goes to
$\,+\,\infty\,$ as $\,r\,$ goes to zero. Furthermore its derivative
$$
g'(r)=\, \frac{1}{r^2}\,\Big(\,\oom(r) -\,\int_0^r
\,\frac{\oom(s)}{s}\,ds\,\Big)%
$$
is strictly negative for positive $\,r\,$ in a neighborhood of the
origin, as follows from the inequality $\,\oom(r) -\,\int_0^r
\,\frac{\oom(s)}{s}\,ds\,<\,0\,,$ for $\,r>\,0\,.$ Let's show this
last inequality. Since the left hand side of the inequality goes to
zero with $\,r\,,$ it is sufficient to show that its derivative is
strictly negative for $\,r>\,0\,.$ This follows easily by appealing
to \eqref{naflitas}. The above results allow us to apply to the
limit \eqref{doente} one of the well known forms of de L'H\^opital's
rule. Straightforward calculations, together with \eqref{naflitas},
show \eqref{peles}.
\end{proof}
Next we state our main results, theorems \ref{sufasvero} and
\ref{sufasvero-2}. In the first theorem constant coefficients are
assumed.
\begin{theorem}
Assume that the oscillation function $\,\oom(r)\,$, concave and
differentiable, satisfies conditions \eqref{massim},
\eqref{simplelas}, and \eqref{naflitas}. Define
$\,{\ho}(r)\,$  by \eqref{chapeu}. Let $\Om_0 \subset \subset \,\Om\,$,
 $\,f \in D_{\oom}(\Ov)\,,$ and $\,u\,$ be the solution of problem \eqref{lapnao},
 where the operator coefficients are constant. Then $\,\na^2\,u \in \,D_{\ho}(\Om_0)\,$ and
\begin{equation}
\|\,\na^2\,u\,\|_{\,\ho,\,\Om_0} \leq \,C\,\|\,f\,\|_{\oom}\,,%
\label{hahega}
\end{equation}
for some positive constant $\,C= C(\Om_0,\,\Om)\,.$  The result is
optimal in the sharp sense defined in section \ref{optimus} .
Furthermore, the above regularity holds up to flat boundary points.
\label{sufasvero}
\end{theorem}
A point $\,x\,\in\,\pa\,\Om\,$ is said to be a \emph{flat boundary
point} if the boundary is flat in a neighborhood of the point. The
meaning of \emph{sharp optimality} is the following (our abbreviate
notation seems clear).
\begin{definition}
We say that a given regularity statement of type $\oom \rightarrow
\,\ho\,$ is sharp if any regularity statement $\oom \rightarrow
\,\ho_0\,,$ obtained by replacing $\,\ho\,$ by any other
$\,\ho_0\,,$ implies the existence
of a constant $c$ for which $\,\ho(r)\leq\,c\,\ho_0(r)\,.$%
\label{sharpas}
\end{definition}
The sharp regularity claimed in theorem \ref{sufasvero} will be
proved in section \ref{optimus}.

\vspace{0.2cm}

Much stronger results hold if the constant $C_1\,$ in equation
\eqref{naflitas} is positive and finite. In this case one has
\begin{equation}
D_{\ho}(\Ov)=\,D_{\oom}(\Ov)\,.%
\label{fullas}
\end{equation}
In fact, by de l'H\^opital rule, one shows that
$$
\lim_{r\rightarrow 0}\,\frac{\ho(r)}{\,\oom(r)}=\,
\lim_{r\rightarrow 0}\,\frac{\oom(r)}{r\,\,\oom'(r)},%
$$
if the second limit exists. Hence, under this last hypothesis, the
identity \eqref{fullas} holds if (actually, and only if) the limit
is positive and finite. Clearly, \eqref{fullas} holds by merely
assuming the inequality required in proposition \ref{eqnormes}. We
will show that if \eqref{fullas} holds then the operator $\bL$ can
have variable coefficients, and full regularity occurs up to any
(regular) boundary point. More precisely, one has the following
result.
\begin{theorem}
Assume that the oscillation function $\,\oom(r)\,$,  concave and
differentiable, satisfies conditions \eqref{massim},
\eqref{simplelas}, and \eqref{naflitas} for some $\,C_1<\,+\,\infty\,$ . Further, define
$\,{\ho}(r)\,$  by \eqref{chapeu}. Let $\,f \in D_{\oom}(\Ov)\,,$
and let $\,u\,$ be the solution of problem \eqref{lapnao}. Then
$\,\na^2\,u \in \,D_{\oom}(\Ov)\,$ and
\begin{equation}
\|\,\na^2\,u\,\|_{\,\oom} \leq \,C\,\|\,f\,\|_{\oom}\,,%
\label{hahega-2}
\end{equation}
for some positive constant $\,C\,.$ Regularity in the sharp sense
holds.%
\label{sufasvero-2}
\end{theorem}
Regularity in the sharp sense follows trivially from full
regularity. But it is quite significant, even necessary, in dealing
with intermediate regularity results, like in theorem
\ref{sufasvero}. See the example shown in section \ref{DOPO}, in the
framework of Log spaces $\,D^{0,\,\al}(\Ov)\,.$%

\vspace{0.2cm}

The conditions imposed in the above statements can be weakened as
follows. We start by replacing the concavity assumption by the
existence of a constant $\,k_1>\,1\,$ such that
\begin{equation}
\oom(k_1\,r) \leq\,c_1\,\oom(r)\,%
\label{comka}
\end{equation}
for some positive constant $\,c_1\,,$ and for $\,r\,$ in a
neighborhood of the origin. We take into account that, if
\eqref{comka} holds, then given $\,k_2>\,1\,,$ there is a positive
constant $\,c_2\,$ such that
\begin{equation}
\oom(k_2\,r) \leq\,c_2\,\oom(r)\,,%
\label{comka-zw}
\end{equation}
for $\,r\,$ in some $\,\de_0-$neighborhood of the origin. The proof
is obvious, by a bootstrap argument. Take into account that, if
$\,k_2>\,k_1\,,$ there is an integer $m$ such that $\,k_2
\leq\,k_1^{\,m}\,.$ If $\,\oom(r)\,$ is concave the lemma
\ref{cocas} shows \eqref{comka} for $\,k_2=\,c_2=\,1\,.$\par%
Actually, in the sequel we will prove that in theorem
\ref{sufasvero}, concavity, differentiability, and assumptions
\eqref{massim}, \eqref{simplelas}, and \eqref{naflitas}, can be
replaced by the more general set of assumptions \eqref{massim},
\eqref{simplelas}, \eqref{comka}, and \eqref{eeles-2}. The same
holds for theorem \ref{hahega-2}, by adding the assumption
\eqref{junto}.

\vspace{0.2cm}

For previous related results we refer to \cite{burch} and
\cite{shapiro}. The author is grateful to Piero Marcati who, after a
seminar on our results, found the above references.
\section{An H-K-L-G inequality.}\label{apotes-G}%
In this section we prove the Theorem \ref{ofundas-G} below. The
proof is an adaptation of that developed in \cite{JBS} to prove the
so called H\"older-Korn-Lichtenstein-Giraud inequality (see
\cite{JBS}, part II, section 5, appendix 1) in the framework of
H\"older spaces. Following \cite{JBS}, we considered \emph{singular
kernels} $\,\mK(x)$ of the form
\begin{equation}
\mK(x)=\,\frac{\sg(x)}{|x|^n}\,,%
\label{kapes}
\end{equation}
where $\,\sg(x)\,$ is infinitely differentiable for $\,x\neq\,0\,$,
and satisfies the properties $\,\sg(t\,x)=\,\sg(x)\,,$ for
$\,t>0\,,$ and
$$
\int_S \sg(x) \,dS =\,0\,,
$$
where $\,S=\,\{\,{x:\,|x|=1}\,\}\,$. It follows easily that, for
$\,0<\,\ro_1<\,\ro_2\,,$%
\begin{equation}
\int_{\ro_1 <|x|< \ro_2} \mK(x) \,dx =\,\int_{\ro_1 <|x|} \mK(x)
\,dx=
\,\int  \mK(x) \,dx=\,0\,,%
\label{simp}
\end{equation}
where the last integral is in the Cauchy principal value sense.\par%
For continuous functions $\,\phi\,$ with compact support, the
convolution integral
\begin{equation}
(\mK \ast \phi)(x)=\,\int \,\mK(x-y)\,\phi(y)\,dy\,,%
\label{convint}
\end{equation}
 extended to the whole space $\,\R^n$, exists as a Cauchy principal value and is finite.%

\vspace{0.2cm}

We set $\,I(\ro)=\{\,x:\,|\,x\,| \leq\, \ro\,\}\,,$
$\,D_{\oom}(\ro)=\,\,D_{\oom}(I(\ro))\,,$ and do the same for other
functional spaces, norms, and semi-norms labeled by $\ro\,$.\par%
\begin{theorem}
Let $\,\mK(x)$ be a singular kernel enjoying the properties
described above. Further, assume that the oscillation function
$\,\oom\,$ satisfies \eqref{massim}, \eqref{simplelas},
\eqref{comka}, and \eqref{eeles-2}. Let $\phi \in
\,D_{\oom}(\ro)\,,$ vanish for $\,|x| \geq\,\ro\,.$ Then $\,\mK \ast
\phi \in\, D_{\ho}(\ro)\,.$ Furthermore, in the sphere $\,I(\ro)\,,$
one has
\begin{equation}
[\,(\mK \ast \phi)\,]_\ho \leq\,C\,\|\,
\phi\,\|_\oom\,,%
\label{tima}
\end{equation}
where $\,C=\,C(n,\,\oom\,,|\|\,\sg\,\||\,)\,.$ \label{ofundas-G}
\end{theorem}
\begin{proof}
Let $x_0,\,x_1 \in I(\ro)\,$, $0<|x_0-\,x_1|=\,\de <\,\de_0
\leq\,\ro\,.$  The positive constant $\,\de_0\,$ is fixed here in
correspondence to the choice  $\,k_2=\,3\,$ in \eqref{comka-zw}. In
the concave case (assumed, for clearness, in the statements of
theorems \ref{sufasvero} and  \ref{sufasvero-2}), we may set
$\,k_2=\,1\,$.\par%
For convenience, we will use the simplified notation
$\,\oom(r)=\,\oom_\phi(r)\,.$ From \eqref{simp} it follows that
$$
(\mK \ast \phi)(x)=\,\int \,\big(\,\phi(y)-\phi(x)\big)\,\mK(x-\,y)
\,dy\,.
$$
Hence, with abbreviated notation,
\begin{equation}
\begin{array}{l}
(\mK \ast \phi)(x_0)\,-(\mK \ast \phi)(x_1)=\\
\\
\int\, \Big\{\,\big(\,\phi(y)-\phi(x_0)\big)\,\mK(x_0-\,y)
\,-\big(\,\phi(y)-\phi(x_1)\big)\,\mK(x_1-\,y)\,\Big\} \,dy=\\
\\
\int_{|y-x_0|<\,2\de} \,\{...\} \,dy +\,\int_{2\de <|y-x_0|<\,\de_0}
\,\{...\} \,dy +\,\int_{\de_0<|y-x_0|} \,\{...\} \,dy \equiv
\,I_1+I_2+I_3\,.
\end{array}
\label{decomp-G}
\end{equation}
Since%
$$
\{y:\,|y-x_1|<\,2\de\} \subset\, \{y: \,|y-x_0|<\,3\de\}
$$
it follows that
\begin{equation}
\begin{array}{l}
\int_{|y-x_0|<\,2\de}
\,\big|\,\phi(y)-\phi(x_1)\big|\,|\mK(x_1-\,y)| \,dy \leq\\%
\\
\int_{|y-x_1|<\,3\de}
\,\big|\,\phi(y)-\phi(x_1)\big|\,|\mK(x_1-\,y)| \,dy\leq \\
\\
\|\,\sg\,\|\,\int_0^{3\de}\,\frac{\oom(r)}{r} \,dr\,\leq\,
\|\,\sg\,\|\,[\,\phi\,]_{\oom} \,\int_0^{3\de}\,\frac{\oom(r)}{r}
\,dr\,,
\end{array}
\label{pora3}
\end{equation}
where we appealed to polar-spherical coordinates with $\,
r=\,|x_1-y|\,,$ to the fact that $\sg$ is positive homogeneous of
order zero, to \eqref{kapes}, and to definition \eqref{fom2}.\par%
A similar, simplified, argument shows that equation \eqref{pora3}
holds by replacing $x_1$ by $x_0$ and $3\de$ by $2\de$. So,
$$
|I_1| \leq \,2\,\|\,\sg\,\|\,[\,\phi\,]_{\oom}
\,\int_0^{3\de}\,\frac{\oom(r)}{r} \,dr\,\leq
\,c\,\|\,\sg\,\|\,[\,\phi\,]_{\oom}\,\,\int_0^{\de}\,\frac{\oom(r)}{r}\,
\,dr\,
$$
where we have appealed to \eqref{comka-zw} for $\,k_2=\,3\,.$ Hence,
\begin{equation}
|I_1| \leq \,c\,\|\,\sg\,\|\,[\,\phi\,]_{\oom}\,\ho(\de)\,.%
\label{pora4-G}
\end{equation}

\vspace{0.2cm}

On the other hand
$$
I_2=\,\int_{2\de <|y-x_0|<\,\de_0}
\,\big(\,\phi(x_1)-\phi(x_0)\big)\,\mK(x_0-\,y) \,dy+
$$
$$
\int_{2\de<|y-x_0|<\,\de_0}
\,\big(\,\phi(y)-\phi(x_1)\big)\,\big(\mK(x_0-\,y)-\,\mK(x_1-\,y)\big)
\,dy\,.
$$
The first integral vanishes, due to \eqref{simp}. Hence,
$$
|I_2| \leq\,\int_{2\de<|y-x_0|<\,\de_0}
\,\big|\,\phi(y)-\phi(x_1)\big|\,\big|\,\mK(x_0-\,y)-\,\mK(x_1-\,y)\,
\big| \,dy\,.
$$
Further, by the mean-value theorem, there is a point $x_2$, between
$x_0$ and $x_1$, such that
$$
\big|\,\mK(x_0-\,y)-\,\mK(x_1-\,y)\,\big|
\leq\,\big|\,\na\,\mK(x_2-\,y)\,\big|\,\de\,.
$$
Since
$$
\pa_i\,\mK(x)=\,\frac{1}{|x|^{n+\,1}}\,\Big[\,(\pa_i\,\sg)
\Big(\frac{x}{|x|}\Big)\,-\,n\,\frac{x_i}{|x|}\,\sg(x)\,\Big]\,,
$$
it readily follows that
\begin{equation}
\begin{array}{l}
\big|\,\mK(x_0-\,y)-\,\mK(x_1-\,y)\,\big| \leq \\
\\
c\,\||\,\sg\,\|| \,\frac{\de}{|y-\,x_2|^{n+1}}
\leq\,c\,\||\,\sg\,\||
\,\frac{\de}{|y-\,x_0|^{n+1}}\,,%
\end{array}
\label{esao2}
\end{equation}
where $\,\||\,\sg\,\||\,$ denotes the sum of the $L^\infty$ norms of
$\sg$ and of its first order derivatives on the surface of the unit
sphere $I(0,1)\,.$ Note that, for $\,|x_0-\,y|>\,2\,\de\,,$ one has
$$
|x_0-\,y|\leq\,2\,|x_2-\,y|\leq\,4\,|x_0-\,y|\,.
$$
On the other hand, for $\,2\,\de <|x_0-\,y|\,,$
$$
|x_1-\,y|\leq\,3\,|x_0-\,y|\,.
$$
So,
$$
|\,\phi(y)-\phi(x_1)\,| \leq\,[\,\phi\,]_{\oom\,} \, \oom
(\,3\,|\,x_0-\,y|\,)\,.
$$
The above estimates show that
\begin{equation}
\begin{array}{l}
|\,I_2\,| \leq\,c\,\||\,\sg\,\||\,[\,\phi\,]_{\oom\,}
\,\de\,\int_{2\de}^{\de_0} \, \oom\big(\,3\,r\,\big)\,
r^{-2} \,dr\\
\\
\leq \,c\,\||\,\sg\,\||\,[\,\phi\,]_{\oom\,}
\,\de\,\int_{2\de}^{\,\de_0} \,\oom(\,r\,)\, r^{-2} \,dr\,,%
\end{array}%
\label{esao2nao}
\end{equation}
where we appealed to \eqref{comka-zw} for $\,k_2=\,3\,.$ Finally, by
\eqref{eeles-2}, it readily follows that
\begin{equation}
|\,I_2\,|\, \leq \,c\,\||\,\sg\,\||\,[\,\phi\,]_{\oom\,}\,\ho(\de)%
\label{esao23-G}
\end{equation}
for $\,\de \in\,(0,\,\de_0\,)\,.$

\vspace{0.2cm}

Finally we consider $I_3$. By arguing as for $I_2$, in particular by
appealing to \eqref{simp} and \eqref{esao2}, one shows that
\begin{equation}
\begin{array}{l}
|I_3| \leq\,C\,\de\,|\|\,\sg\,\||\,\int_{|y-x_0|>\,\de_0} \,
\,\frac{|\,\phi(y)-\phi(x_1)\big|}{\,|y-\,x_0|^{n+1}} \; dy \leq \\
\\
C\,\de\,|\|\,\sg\,\||\,\|\,\phi\,\|
\leq\,C\,|\|\,\sg\,\||\,\|\,\phi\,\|\,\ho(\de)\,.%
\end{array}%
\label{thislast}
\end{equation}
Note that, by a  de l'H\^opital rule, one shows that
\eqref{simplelas} holds with $\,\oom(r)\,$ replaced by $\,\ho(r)\,.$
From equation \eqref{decomp-G}, by appealing to \eqref{pora4-G},
\eqref{esao23-G}, and \eqref{thislast}, one shows that
\begin{equation}
|\,(\mK \ast \phi)(x_0)\,-(\mK \ast
\phi)(x_1)\,|\leq\,C\, \||\,\sg\,\|| \,\|\,\phi\,\|_{\oom\,}\,\ho(\de)\,,%
\label{tresis}
\end{equation}
for each couple of points $\,x_0,\,x_1 \in\,I(\ro)\,$ such that
$\,0<\,|x_0-\,x_1\,| \leq\,\de_0\,.$ Hence \eqref{ofundas-G}
holds.\par%
We may easily estimate $|\,(\mK \ast \phi)(x_0)\,-(\mK \ast
\phi)(x_1)\,|\,$ for pairs of points $\,x_0,\,x_1\,$ for which
$\,\de_0<\,|x_0-\,x_1\,| < \ro\,.$ However this is superfluous,
since $\,\de_0\,$ is fixed "once and for all".
\end{proof}
\section{The interior regularity estimate in the constant coefficients case.}\label{elipse}%
In this chapter we apply the theorem \ref{ofundas-G} to prove the
basic interior regularity result for solutions of the elliptic
equation \eqref{lapnao} in the framework of $\,D_{\oom}\,$ data
spaces. In this section $\bL$ ia a constant coefficients operator.
The proof is inspired by that developed in H\"older spaces in
\cite{JBS}, part II, section 5. For convenience, assume that $n\geq\,3\,.$%

\vspace{0.2cm}

By a fundamental solution of the differential operator $\,\bL\,$ one
means a distribution $\,J(x)\,$ in $\,\R^n\,$ such that
\begin{equation}
\bL\,J(x)=\,\de(x)\,.%
\label{fundas}
\end{equation}
The celebrated Malgrange-Ehrenpreis theorem states that every
non-zero linear differential operator with constant coefficients has
a fundamental solution (see, for instance, \cite{yosida}, Chap. VI,
sec. 10). We recall that the analogue for differential operators
whose coefficients are polynomials (rather than constants)
is false, as shown by a famous Hans Lewy's counter-example.\par%
In particular, for a second order elliptic operator with constant
coefficients and only higher order terms, one can construct
explicitly a fundamental solution $\,J(x)$ which satisfies the
properties (i), (ii), and (iii), claimed in \cite{JBS}, namely,\par%
(i)  $J(x)$ is a real analytic function for $\,|x| \neq\,0\,.$\par%
(ii) For $n\geq\,3\,$
\begin{equation}
J(x)=\,\frac{\sg(x)}{|x|^{n-\,2}}\,,%
\label{jicas}
\end{equation}
where $\sg(x)$ is positive homogeneous of degree $\,0\,$.\par%
(iii) Equation \eqref{fundas} holds. In particular, for every
sufficiently regular, compact supported, function $\,v\,$, one has
$$
v(x)=\,\int \,J(x-\,y)\,(\bL\,v)(y) \,dy\,.
$$
For a second order elliptic operator as above, one has%
\begin{equation}
J(x)=\,c\,\big(\,\sum\, A_{i\,j}\,x_i x_j\,\big)^{\frac{2-\,n}{2}}\,,%
\label{janota}
\end{equation}
where $A_{i\,j}$ denotes the cofactor of $a_{i\,j}$ in the
determinant $|\,a_{i\,j}\,|\,.$\par%
Following \cite{JBS}, we denote by $\bS$ the operator
\begin{equation}
(\bS\,\ph)(x)=\,\int \,J(x-\,y)\,\ph(y) \,dy=\,(J\ast\ph)(x)\,.%
\label{beesse}
\end{equation}
Note that, in the constant coefficients case, the
operator $\bT$ introduced in reference \cite{JBS} vanishes.\par%
Point (iii) above (see also \cite{JBS} "Lemma" A) shows that if $v$
is compact supported and sufficiently regular (for instance of class
$\,C^2\,$), then
\begin{equation}
v=\, \bS \bL\,v\,.%
\label{slfi}
\end{equation}
Due to the structure of the function $\,\sg(x)\,$ appearing in
equation \eqref{jicas}, it readily follows that second order
derivatives of $(\bS\,\ph)(x)$ have the form $\,\pa_i\,\pa_j
\,\bS\,\ph=\,\mK_{i\,j} \ast \phi\,,$ where the $\mK_{i\,j}$ enjoy
the properties described for singular kernels $\mK$ in section \ref{apotes-G}.\par%
We write, in abbreviated form,
\begin{equation}
\na^2\,\bS\,\ph\,(x)=\,\int \,\mK(x-y)\,\phi(y)\,dy\,,%
\label{convintos}
\end{equation}
where $\,\mK(x)\,$ enjoys the properties described at the beginning
of section \ref{apotes-G}. From  \eqref{convintos} it follows that
$$
\na^2\,\bS \bL v=\,\int \,\mK(x-y)\,\bL\,v(y)\,dy\,.%
$$
Hence, by Theorem \ref{ofundas-G}, one gets
\begin{equation}
[\,\na^2\,\bS \bL v\,]_{\ho;\,2\,\ro}\leq\,C\,[\,
\bL \,v\,]_{\oom;\,2\,\ro}\,.%
\label{masimus}
\end{equation}
By appealing to \eqref{slfi} we get the following result.
\begin{proposition}
Assume that the differential operator $\,\bL\,$ has constant
coefficients and that the oscillation function $\,\oom\,$ satisfies
assumptions \eqref{massim}, \eqref{simplelas}, \eqref{comka}, and
\eqref{eeles-2}. Let $\,v\,$ be a support compact function $ \in
C^{2}(2\,\ro)\,,$ such that $\,\bL\,v \in D_{\oom}(2\,\ro)\,.$ Then
\begin{equation}
[\na^2\, v\,]_{\ho;\,2\,\ro}\leq\,C\,[\,
\bL \,v\,]_{\oom;\,2\,\ro}\,.%
\label{larj}
\end{equation}
\label{lemitas}
\end{proposition}

\vspace{0.2cm}

One has the following interior regularity result. For brevity we
have consider two spheres of radius $\ro$ and $R$,  $R>\,\ro$, in
the particular case $R=\,2\,\ro.$
\begin{theorem}
Assume that the hypothesis of proposition \ref{lemitas} hold.
Further, let $u \in C^{2}(2\,\ro)\,$ be such that $\,\bL\,u \in
D_{\oom}(2\,\ro)\,.$ Then $\,\na^2\,u \in \,D_{\ho}(\ro)\,,$
moreover
\begin{equation}
[\na^2\,u\,]_{\,\ho;\,\ro} \leq \,C\,[\,\bL\,u\,]_{\oom,\,2\,\ro}+
\,c(\te)\,\Big(\,\frac{\|u\|}{\ro^3}+\,\frac{\|\,\na\,u\|}{\ro^2}+\,\frac{\|\,\na^2\,u\|}{\ro}\,\Big)\,\frac{|x-\,y|}{\oom(|x-\,y|)}\,,
\label{acasisim}
\end{equation}
for some positive constant $\,C,$ independent of $\ro\,$. In particular,
\begin{equation}
[\na^2\,u\,]_{\,\ho;\,\ro} \leq \,C\,[\,\bL\,u\,]_{\oom,\,2\,\ro}+
\,\frac{c(\te)}{\ro^3} \,\|\,u\|_{C^2(2\ro)}\,.
\label{acasisim-2}
\end{equation}
%VECCHIO:
%\begin{equation}
%\|\,\na^2\,u\,\|_{\,\ho} \leq \,C\,\|\,f\,\|_{\oom}\,,%
%\label{suba}
%\end{equation}
\label{ajudas}
\end{theorem}
\begin{proof}
 Fix a no-negative $C^\infty\,$
function $\te$, defined for $\,0\leq t \leq 1\,$ such that
$\te(t)=\,1$ for $\,0\leq t \leq \frac13\,,$ and $\te(t)=\,0$ for
$\,\frac23\leq t \leq 1\,.$ Further fix a positive real $\,\ro\,$, for
convenience $0<\ro<\,\frac12\,,$ and define
\begin{equation}
\z(x)= \left\{
\begin{array}{l}
1 \quad \textrm{for} \quad |x|\leq \ro\,,\\
\\
\te \big(\frac{|x|-\,\ro}{\ro}\big) \quad \textrm{for} \quad \ro\leq
|x|\leq 2\,\ro\,.
\end{array}
\right.%
\label{zetaze}
\end{equation}
%============
Next we consider $\,\z(x)\,$ for points $\,x$ such that $\,\ro\leq\,|x|\leq\,2\,\ro\,,$
and leave to the reader different situations. Due to
symmetry, it is sufficient to consider the one dimensional case
$$
\z(t)=\,\te \big(\frac{t-\,\ro}{\ro}\big) \quad \textrm{for} \quad \ro
\leq t\leq 2\,\ro\,.
$$
Hence
$$
\z'(t)=\,\te' \big(\frac{t-\,\ro}{\ro}\big)\,\frac{1}{\ro}\,,
$$
and
$$
\z''(t)=\,\te'' \big(\frac{t-\,\ro}{\ro}\big)\,\frac{1}{\ro^2}\,.
$$
Further,
$$
\ro^2\, |\z''(t_2)-\z''(t_1)| \leq\,\Big| \, \te''
\Big(\frac{t_2-\,\ro}{\ro}\Big)
-\,\te''\Big(\frac{t_1-\,\ro}{\ro}\Big)\,\Big|\,,
$$
where
$$
\Big|\,\frac{t_2-\,\ro}{\ro}-\,\frac{t_1-\,\ro}{\ro}\,\Big|=\,\Big|
\frac{t_2-\,t_1}{\ro}\Big|\leq\,\frac13 <\,1\,.
$$
So
\begin{equation}
|\z''(t_2)-\z''(t_1)\,| \leq  \, \frac{1}{\ro^3}\,
\,[\,\te''\,]_{Lip}\,|\,|t_2-\,t_1|\,,
\label{zedois}
\end{equation}
 where $\,[\,\cdot\,]_{Lip}\,$ denotes the usual Lipschitz semi-norm.

\vspace{0.2cm}

Set
\begin{equation}
v=\,\z\,u\,.
\label{vezu}
\end{equation}
Note that $\,\bL\,v \in D_{\oom}(2\ro)\,,$ moreover the support of
$\,v\,$ is contained in $\,|x|<\,2\ro\,.$\par%
On the other hand,
\begin{equation}
\bL v=\,\z  \bL u+\, N\,.
\label{lvn}
\end{equation}
One has
$$
\begin{array}{l}
|\,(\z \bL u)(x)-\,(\z \bL u)(y)\,|  \leq\,\|\z\|\,[ \bL u ]_{\oom}\,\oom(|x-\,y|) +\,\|\,\na\,\z\,\|\,\| \bL u\,\|\,|x-\,y| \\
\\
 \leq \,[ \bL u ]_{\oom}\,\oom(|x-\,y|) +c\,\|\te' \|\,\frac{1}{\ro}\,\|\na^2\,u\|\,|x-\,y|\,.
\end{array}
$$
Hence,
\begin{equation}
[\,\z \bL u\,]_{\oom}  \leq\, [ \bL u ]_{\oom}+c\,\|\te' \|\,\frac{1}{\ro}\,\|\na^2\,u\|\,\frac{|x-\,y|}{\oom(|x-\,y|)}\,.
\label{princ}
\end{equation}
Next we prove that
\begin{equation}
[\,N\,]_\oom \leq\,c(\te)\,\Big(\,\frac{\|u\|}{\ro^3}+\,\frac{\|\,\na\,u\|}{\ro^2}+\,\frac{\|\,\na^2\,u\|}{\ro}\,\Big)\,\frac{|x-\,y|}{\oom(|x-\,y|}\,.
\label{eanes}
\end{equation}
One has
$$
N\cong (\na^2 \z)\,u+\,(\na\,\z)\,(\na\,u)\equiv A+\,B\,.
$$
By appealing in particular to \eqref{zedois}, straightforward calculations show that
$$
\begin{array}{l}
|A(x)-A(y)|\leq\,\|\na u\|\,\|\na^2\,\z\|\,|x-\,y|+\,\|u\| \frac{1}{\ro^3}[\te'']_{Lip} |x-\,y|\\
\\
\leq\,\Big( \frac{1}{\ro^2} \|\te''\|\,\|\na u\|+\, \frac{1}{\ro^3} [\te'']_{Lip}    \,\| u\|\,\Big)\,|x-y|\,.
\end{array}
$$
Hence
\begin{equation}
[\,A\,]_\oom \leq\,c(\te)\,\Big(\,\frac{\|u\|}{\ro^3}+\,\frac{\|\,\na\,u\|}{\ro^2}\,\Big)\,\frac{|x-\,y|}{\oom(|x-\,y|}\,.
\label{aass}
\end{equation}
Similar manipulations show that
\begin{equation}
[\,B\,]_\oom \leq\,c(\te)\,\Big(\,\frac{\|\,\na\,u\|}{\ro^2}+\,\frac{\|\,\na^2\,u\|}{\ro}\,\Big)\,\frac{|x-\,y|}{\oom(|x-\,y|}\,.
\label{bbss}
\end{equation}
Equation  \eqref{eanes} follows from \eqref{aass} and \eqref{bbss}\,. \par%
Lastly, from \eqref{lvn}, \eqref{princ}, and \eqref{eanes} one shows that
\begin{equation}
[\,\bL\,v\,]_\oom \leq \,[\,\bL\,u\,]_\oom+
\,c(\te)\,\Big(\,\frac{\|u\|}{\ro^3}+\,\frac{\|\,\na\,u\|}{\ro^2}+\,\frac{\|\,\na^2\,u\|}{\ro}\,\Big)\,\frac{|x-\,y|}{\oom(|x-\,y|}\,.
\label{tormes}
\end{equation}
In the following not labeled norms concern the domain $\,I(2\,\ro)\,.$\par%
From \eqref{vezu}, \eqref{slfi}, \eqref{masimus}, and \eqref{tormes} one gets
\begin{equation}
\begin{array}{l}
[\na^2\,u\,]_{\,\ho;\,\ro} \leq\,[\na^2\,v\,]_{\,\ho}
\leq\,C\,[\,\bL\,v\,]_{\oom}\\
\\
 \leq \,C\,[\,\bL\,u\,]_\oom+
\,c(\te)\,\Big(\,\frac{\|u\|}{\ro^3}+\,\frac{\|\,\na\,u\|}{\ro^2}+\,\frac{\|\,\na^2\,u\|}{\ro}\,\Big)\,\frac{|x-\,y|}{\oom(|x-\,y|}\,,
\end{array}
\label{acasis}
\end{equation}
where $\,0<2\,\ro<1\,.$
\end{proof}
\section{The interior regularity estimate in the variable coefficients case.}\label{muda}%
In this section we extend the estimate \eqref{acasisim} to uniformly
elliptic operators with variable coefficients
\begin{equation}
\bL=\,\sum_1^n a_{i\,j}(x) \pa_i\,\pa_j\,.%
\label{elld}
\end{equation}
To avoid non significant manipulations we assume that the
coefficients $\, a_{i\,j}(x) $ are Lipschitz continuous in
$\,I(2\,\ro)\,,$ which Lipschitz constants bounded by a constant
$\,A\,$. Following the same belief,
we left to the reader the introduction of lower order terms.\par%
We assume that
\begin{equation}
\oom(r)\leq\,k_1\,\ho(r)\,,
\label{junto}
\end{equation}
for some positive constant $k_1$, and $\,r\,$ in some neighborhood
of the origin. This yields $\,D_{\oom}(\Ov)=\,D_{\ho}(\Ov)\,,$
recall proposition \ref{eqnormes}. Assumption \eqref{junto} holds if
in equation \eqref{naflitas} the constant $C_1$ is finite. In fact,
$$
\lim_{r\rightarrow 0}\,\frac{\oom(r)}{\,\ho(r)}=\lim_{r\rightarrow
0}\,\frac{r\,\,\oom'(r)}{\oom(r)}=\,\frac{1}{C_1}\,,
$$
if the second limit exists.\par%
In the following we appeal to the constant coefficients operator
\begin{equation}
\bL_0=\,\sum_1^n  b_{i\,j} \pa_i\,\pa_j\,,%
\label{ellez}
\end{equation}
where  $\,b_{i\,j}=\, a_{i\,j}(0)\,.$ Clearly,
\begin{equation}
\bL_0 v(x)=\,\bL v(x)+ (\,\bL_0 -\,\bL) v(x)\,.
\label{zerob}
\end{equation}
One has
\begin{equation}
\begin{array}{l}
(\,\bL_0-\,\bL)v(x)-\,(\,\bL_0-\,\bL)v(y)=\\
\\
\,(\,( b_{i\,j}-\, a_{i\,j}(x)\,)\,(\,\pa^2_{i\,j}\,v(x)-\,\pa^2_{i\,j}\,v(y)\,)
+\,(\,( a_{i\,j}(y)-\, a_{i\,j}(x)\,)\,(\,\pa^2_{i\,j}\,v(y)
\end{array}
\label{semn}
\end{equation}
where, for convenience, summation on repeated indexes is assumed. Straightforward
calculations easily lead to the following pointwise estimate
\begin{equation}
|\,(\,\bL_0-\,\bL)v(x)-\,(\,\bL_0-\,\bL)v(y)\,| \leq\, c\,A
\big(\,2\,\ro\,[\,\na^2\,v\,]_\oom
+\,\|\na^2\,v\,\|\,\frac{|x-y|}{\oom(|x-y|)}\,\big) \,\oom(|x-y|)\,,
\label{ellez}
\end{equation}
where norms and semi-norms concern the sphere $\,I(0,\,2\,\ro)\,$.\par%
Next assume that $\,v\in C^{2}(2\,\ro)\,$  has compact support in
$\,I(0,\,2\,\ro)\,$, and  $\,\bL\,v \in D_{\oom}(2\,\ro)\,.$  Then,
by \eqref{zerob}, \eqref{semn}, and \eqref{larj} it follows that
$$
[\na^2\, v\,]_{\ho;\,2\,\ro}\leq\,C\,[\, \bL
\,v\,]_{\oom;\,2\,\ro}\,+\,C\,\ro\,[\,\na^2\,v\,]_{\oom;\,2\,\ro}
+\, C\,\|\na^2\,v\,\|\,\frac{|x-y|}{\oom(|x-y|)}\,.
$$
In particular
\begin{equation}
[\na^2\, v\,]_{\ho;\,2\,\ro}\leq\,C\,[\, \bL
\,v\,]_{\oom;\,2\,\ro}\,+\,C\,\ro\,[\,\na^2\,v\,]_{\oom;\,2\,\ro}
+\, C\,\|\na^2\,v\,\|\,.%
\label{seigual}
\end{equation}
Now, from \eqref{junto}, one gets
\begin{equation}
(\,1-C\,k_1\,\ro\,)\,[\na^2\, v\,]_{\oom;\,2\,\ro}\leq\,C\,(\,[\,
\bL \,v\,]_{\oom;\,2\,\ro} +\,\|\na^2\,v\,\|\,).%
\label{seigual}
\end{equation}
Next we set
$$
v=\,\z\,u
$$
and argue as done to prove \eqref{acasisim}\,. This proves the
following result, in the case of variable coefficients operators.
\begin{theorem}
Assume that the oscillation function $\,\oom\,$ satisfies conditions
\eqref{massim}, \eqref{simplelas}, \eqref{comka}, \eqref{eeles-2},
and \eqref{junto}. Further, assume that
$$
0<\,\ro \leq\,\frac{1}{2 C k_1}\,,
$$
and let $\,\bL\,u \in D_{\oom}(2\,\ro)\,,$ for some $u \in
C^{2}(2\,\ro)\,.$ Then $\,\na^2\,u \in \,D_{\ho}(\ro)\,,$ and
\begin{equation}
[\na^2\,u\,]_{\,\ho;\,\ro} \leq \,C\,[\,\bL\,u\,]_{\oom,\,2\,\ro}+
\,\frac{C}{\ro^3} \,\|\,u\|_{C^2(2\ro)}\,,%
\label{acasopas-2}
\end{equation}
for suitable positive constants $\,C\,,$ independent of $\ro\,$.
\end{theorem}
\section{Proof of theorems \ref{sufasvero} and \ref{sufasvero-2}.}\label{elipse2}%
The local estimates (estimates in $\,\Om_0 \,,$ $\,\Om_0 \subset
\subset\,\Om\,$) claimed in theorems \ref{sufasvero} and
\ref{sufasvero-2} follow immediately from the interior estimates, by
appealing to the classical method consisting in covering $\Ov_0\,$
by a finite number of sufficiently small spheres. For brevity, we
may estimate the quantities originated by the terms
$\|\,u\|_{C^2(2\ro)}\,,$ see the right hand sides of equations
\eqref{acasisim-2} and \eqref{acasopas-2}, simply by appealing to
the theorem \ref{laplaces}, which shows that solutions $u$ satisfy
the estimate
\begin{equation}
\|\,u\,\|_{C^2(\Ov)} \leq \,c\,\|\,f\,\|_*\,.%
\label{lapili-2}
\end{equation}
Concerning regularity up to the boundary one proceeds as follows.
The main point, the extension of the interior regularity estimate
\eqref{acasisim-2} from spheres to half-spheres, is obtained by
following the argument described in part II, section 5.6, reference
\cite{JBS}. One starts by showing that the interior estimate in
spheres also hold for half-spheres, under the zero boundary
condition on the flat part of the boundary. One appeals to
"reflection" of $\,u\,$ through the flat boundary, as an odd
function, in the orthogonal direction, from the half to the whole
sphere. In this way the half-sphere problem goes back to an
whole-sphere problem, absolutely similar to that considered in
section \ref{elipse}, see \cite{JBS}. Note tat is is sufficient, and
more convenient, to show this extension to half-spheres merely for
constant coefficient operators. The regularity result "up to flat
boundary points" claimed in theorem \ref{sufasvero} follows.\par%
Extension of the half-sphere's estimate, from constant coefficients
to variable coefficients, is obtained exactly as done in section
\ref{muda} for whole spheres. Obviously, this requires assumption
\eqref{junto}. Then, sufficiently small neighborhoods of boundary
points are regularly mapped, one to one, onto half-spheres. This
procedure allows extension of the local estimate for functions
$\,u\,$ defined on sufficiently small neighborhoods of boundary
points, vanishing on the boundary. A well known finite covering
argument leads to the thesis of theorem \ref{sufasvero-2}.\par%
The above extension to non-flat boundary points requires local
changes of coordinates. This transforms constant coefficients in
variable coefficients operators. So, local regularity up to non-flat
boundary points for constant coefficients operators can not be
claimed here. This is a challenging open problem. One may start by
considering the particular case of data in Log spaces.

\vspace{0.2cm}

We note that in the proof of Theorem 1, section 5.4, part II, in
reference \cite{JBS}, density of $\,C^1\,$ in H\H older spaces is
used. The same occurs in the proof of lemma B, section 5.3.
\section{The Log spaces $ D^{0,\,\al}(\Ov)\,.$ An intermediate regularity result. }\label{DOPO}
The following is a significant example of functional space
$\,D_{\oom}(\Ov)\,$ which yields intermediate (not full) regularity,
based on the well known formulae
\begin{equation}
\int \,\frac{(-\log{r})^{-\,\al}}{r} \,dr
=\,\frac{1}{\al-\,1}\,(-\log{r})^{1-\,\al}\,,%
\label{simvales}
\end{equation}
where $\,0<\,\al<\,+\infty\,$ (for $\,\al=\,1\,$ the right hand side
should be replaced by $\,-\log\,(-\log{r})\,).$ Equation
\eqref{simvales} shows that the $\,C_*(\Ov)\,$ semi-norm
\eqref{seis} is finite if
\begin{equation}
\om_f(r)\leq\, C\,(-\log{r})^{-\,\al}\,,%
\label{alfas}
\end{equation}
for some $\,\al >\,1\,$ and some constant $\,C>\,0\,.$ This led to
define, for each fixed $\,\al>\,0\,,$ the semi-norm
\begin{equation}
[\,f\,]_{\al} \equiv\,
\sup_{\,r \in (0,\,1) }\,\frac{\om_f(r)}{\om_\al(r)} \,,%
\label{alfas2}
\end{equation}
where the \emph{oscillation function} $\,\om_\al(r)\,$ is defined by
setting
\begin{equation}
\om_\al(r) =\,(-\log{r})^{-\,\al}\,.%
\label{alfaerre}
\end{equation}
Hence $\,[\,f\,]_{\al}\,$ is the smallest constant for which the
estimate
\begin{equation}
|f(x)-\,f(y)\,|\leq\, [\,f\,]_{\al} \,\cdot\,\Big(\,
\log\,\frac{1}{|\,x-\,y|}\,\Big)^{-\,\al}%
\label{alforreca}
\end{equation}
holds for all couple $\,x,\,y \in\,\Ov\,$ such that
$\,|x-\,y|<\,1\,.$ Note that we have merely replaced, in the
definition of H\"older spaces, the quantity
$$
\frac{1}{|\,x-\,y|} \quad \textrm{ by}  \quad
\log{\frac{1}{|\,x-\,y|}}\,,
$$
and allow $\,\al\,$ to be arbitrarily large.
\begin{definition}
For each real positive $\,\al\,,$ we set
\begin{equation}
D^{0,\,\al}(\Ov) \equiv\,\{\,f \in\,C(\,\Ov): \,[\,f\,]_{\al}
<\,\infty\,\}\,.%
\label{cstar}
\end{equation}
A norm is introduced in $\,D^{0,\,\al}(\Ov)\,$ by setting $
\,\|\,f\,\|_{\al}\equiv\,[\,f\,]_{\al}+\,\|\,f\,\|\,. $%
\label{defcstar}
\end{definition}
We call these spaces Log spaces. We remark that in reference
\cite{BV-arxiv}
we have called these spaces H-log spaces. \par%
The restriction $\,|x-\,y|<\,1\,$ in equation \eqref{alfas2} is due
to the behavior of the function $\,\log{r}\,,$ for $\,r \geq\,1\,.$
Note that, by replacing $\,0<\,|x-\,y|<\,1\,$ by
$\,0<\,|x-\,y|<\,\ro\,$ in equation \eqref{alfas2}, for some
$\,0<\,\ro<\,1\,,$ it follows that
\begin{equation}
[\,f\,]_{\al;\,\ro} \le\,[\,f\,]_{\al}
\leq\,[\,f\,]_{\al;\,\ro}+\,\frac{2}{(-\log{\ro})^{-\,\al}}\,\|\,f\,\|\,,
\label{clearnot}
\end{equation}
where the meaning of $\,[\,f\,]_{\al;\,\ro}\,$ seems clear. Hence,
the norms $\,\|\,f\,\|_{\,\al}\,$ and $\,\|\,f\,\|_{\al;\,\ro}\,$
are equivalent. We may also avoid the above $\,|x-\,y|<\,1\,$
inconvenient by replacing in the denominator of the right hand side
of \eqref{alfas2} the quantity $\,r\,$ by $\,r/R\,$, where $\,R
=\,\textrm{diam}\,\Om\,$, and by letting $r\in (0,\,R)\,.$ We rather
prefer the first definition, since the second one implies more ponderous notation.\par%
For $\,0<\,\bt <\,\al\,,$ and $\,0<\,\la \leq\,1\,,$ the (compact)
embedding
\begin{equation}
\quad C^{0,\,\la}(\Ov)\subset D^{0,\,\al}(\Ov) \subset D^{0,\,\beta}(\Ov)\, \subset \,C(\Ov)%
\label{oitooos}
\end{equation}
hold. Furthermore,  for $\,1<\,\al\,,$ one has the (compact)
embedding $\,
 D^{0,\,\al}(\Ov) \subset \,C_*(\Ov)\,.\,$  Note that $\,D^{0,\,1}(\Ov) \nsubseteq
\,C_*(\Ov)\,$.\par%
The properties proved in reference \cite{BV-arxiv} for $\,
D^{0,\,\al}(\Ov)\,$ spaces follow here from that proved for
$\,D_\oom (\Ov)\,$ spaces, since $\,\om_\al(r)\,$ is a particular
case of function  $\,\oom(r)\,.$ It is worth noting that in
reference \cite{BV-arxiv} we claimed, and left the proof to the
reader, that $\,C^{\infty}(\Ov)\,$ is dense in $\,
D^{0,\,\al}(\Ov)\,$. Actually, as shown in theorem \ref{osned}, this
result is false.\par%
The following result is a particular case of theorem
\ref{sufasvero}.
\begin{theorem}
Let $\Om_0 \subset \subset \,\Om\,$, $\,f\in\, D^{0,\,\al}(\Ov)\,$
for some $\,\al>\,1\,,$ and $\,u\,$ be the solution of problem
\eqref{lapnao}, where $\bL$ has constant coefficients. Then
$\,\na^2\,u \in\, D^{0,\,\al-\,1}(\Om_0)\,,$ moreover
\begin{equation}
\|\,\na^2\,u\,\|_{\,\al-1,\,\Om_0} \leq \,C\,\|\,f\,\|_\al\,,%
\label{hahega}
\end{equation}
for some positive constant $\,C= C(\al,\,\Om_0,\,\Om)\,.$ The
regularity result holds up to flat boundary points. Results are
optimal in the sharp sense, see section \ref{optimus}. In
particular,for $\,\beta >\,\al-1\,,$ $\,\na^2\,u \in\,
D^{0,\,\beta}(\Om_0)\,$ is false in general.%
\label{laplolas}
\end{theorem}
Theorem \ref{laplolas} is a particular case of theorem
\ref{sufasvero}. In fact, the oscillation function $\,\om_\al(r)\,$
is concave and differentiable for $\,r>\,0\,,$ satisfies
\eqref{massim} for $\,\al> \,1\,,$ and \eqref{simplelas} holds.
Further, condition \eqref{naflitas} follows from
\begin{equation}
\lim_{r\rightarrow 0}\,\frac{\om_\al(r)}{r\,\,\om'_\al(r)} =\,+
\,\infty\,.%
\label{biri}
\end{equation}
In reference \cite{BV-arxiv} the above regularity result was claimed
up to the boundary. However the proof is not complete, since
extension to non-flat boundary points requires estimates for
variable coefficients operators. The reason for this requirement was
explained in section \ref{elipse2}.

\vspace{0.2cm}

Next we apply to the results stated in theorem \ref{laplolas} to
illustrate, by means of a simple example, the meaning of
\emph{sharp} optimality. This concept will be discussed in a more
abstract form in section \ref{optimus}. Optimality of regularity
results is not confined here to the particular family of spaces
under consideration, but is something stronger. Let us illustrate
the distinction. The theorem \ref{laplolas} claims that
$\,\na^2\,u\in\,D^{0,\,\al-\,1}(\Ov)\,$. Optimality
\emph{restricted} to the Log spaces framework means that, given
$\,\beta>\,\al-\,1\,,$ there is at least a data  $\,f\in\,
D^{0,\,\al}(\Ov)\,$ for which $\,\na^2\,u\,$ does not belong to
$\,D^{0,\,\beta}(\Ov)\,.$ This situation does not exclude that (for
instance, and to fix ideas) for all $\,f\in\,D^{0,\,\al}(\Ov)\,$ the
oscillation $\,\om(r)\,$ of $\,\na^2\,u\,$ satisfies the stronger
estimate
\begin{equation}
\om(r)\leq\, C_f\,\big[\,\log
\big(\log{\frac1r}\,\big)\big]^{-1}\,\cdot\,(-\log{r})^{-\,(\al-1)}\,.%
\label{alfasin}
\end{equation}
In fact, for each $\,\beta>\,\al-\,1\,,$ one has
$$
(-\log{r})^{-\,\beta} << \,\big[\,\log
\big(\log{\frac1r}\,\big)\big]^{-1}\,\cdot\,(-\log{r})^{-\,(\al-1)}
\,<<\,(-\log{r})^{-\,(\al-1)}\,.%
$$
Sharp optimality avoids the above, and similar, possibilities. This
fact is significant in all cases in which full regularity is not
reached, as in Theorem \ref{laplolas}. This is the meaning giving
here to the sharpness of a regularity result.

\vspace{0.2cm}

Concerning references, not related to our results but merely to Log
spaces (mostly for $n=\,1\,$, or $\al=\,1$), the author is grateful
to Francesca Crispo for calling our attention to the treatise
\cite{fiorenza}, to which the reader is referred. In particular, as
claimed in the introduction of this volume, the space
$\,D^{0,\,1}(\Ov)\,$ was considered in reference \cite{shara}. See
also definition 2.2 in reference \cite{fiorenza}. Other references,
quoted in \cite{fiorenza}, are \cite{diening}, \cite{samko},
\cite{zhikov1}, \cite{zhikov2}, and \cite{zhikov3}.
\section{ H\"olog spaces $\,C^{0,\,\la}_\al(\Ov)\,$ and full regularity.}\label{hologos}
If, for some $\,\la>\,0\,,$ one has $\,\ho(r)=\,\la\,\oom(r)\,$ in a
neighborhood of the origin, then there is a $\,k>\,0\,$ such that
$\,\oom(r)=\,k\,r^\la \,.$ This fact could suggest that H\"older
spaces could be the unique full regularity class inside our
framework. However, \emph{full regularity} is also enjoyed by other
spaces. The following is a particularly interesting case. Consider
oscillation functions
\begin{equation}
 \om_{\la,\,\al}(r)=\,r^\la \,(-\log{r})^{-\,\al}\,, \quad \,r<\,1\,,%
\label{lalfas}
\end{equation}
where $\,0\leq\,\la <\,1\,$ and $\,\al \in\,\R\,.$  For
$\,\la=\,0\,$ and $\,\al>\,0\,$ we re-obtain the Log space
$\,D^{0,\,\al}(\Ov)\,$, for $\,\la>\,0\,$ and $\,\al=\,0\,$ we
re-obtain $\,C^{0,\,\la}(\Ov)\,$. Theorem \ref{compactos} shows that
(compact) inclusions
$$
C^{0,\,\la_2}(\Ov)\subset \,C^{0,\,\la}_\al(\Ov)\subset
\,C^{0,\,\la}_\beta(\Ov) \subset\,C^{0,\,\la}(\Ov)\subset
\,C^{0,\,\la}_{-\beta}(\Ov) \subset \,C^{0,\,\la}_{-\al}(\Ov)
\subset\,C^{0,\,\la_1}(\Ov)%
$$
hold for $\,\al>\,\beta>\,0\,$ and
$\,0<\,\la_1<\,\la<\,\la_2<\,1\,.$ The reader should note that the
set
$$
\bigcup_{\la,\,\al} \,C^{0,\,\la}_\al(\Ov)\,,
$$
where $\,0<\la<\,1\,,$ and $\,\al \in\,\R\,,$ is a \emph{totally}
ordered set, in the obvious way. In the totally ordered sub-chain
merely consisting of classical H\H older spaces, each
$\,C^{0,\,\la}\,$ space can be enlarged, to became an infinite,
ordered chain, $\,C^{0,\,\la}_\al(\Ov)\,$,  $\,\al \in\,\R\,.$
Clearly, the spaces $\,C^{0,\,\la}_\al(\Ov)\,$  enjoy all the
interesting properties described in section \ref{novicas}.\par%
To abbreviate notation, in this section we set
$$
\oom(r)\equiv \om_{\la,\,\al}(r)\,, \quad  [\,f\,]_{\,\oom} \equiv
[\,f\,]_{\la,\,\al}\,, \quad \textrm{and} \quad  \|\,f\,\|_{\,\oom}
\equiv  \|\,f\,\|_{\la,\,\al}\,.
$$
The following full regularity result holds.
\begin{theorem}
Let $\,f\in\,C^{0,\,\la}_\al(\Ov)\,$ for some $\,\la \in
\,(\,0,\,1\,)\,$ and some $\,\al\in \R\,.$ Let $\,u\,$ be the
solution of problem \eqref{lapnao}, where the differential operator
$\,\bL\,$ may have variable coefficients. Then $\,\na^2\,u
\in\,C^{0,\,\la}_\al(\Ov)\,.$ Moreover
\begin{equation}
\|\,\na^2\,u\,\|_{\la,\,\,\al} \leq \,C\,\|\,f\,\|_{\la,\,\,\al} \,,%
\label{haligas}
\end{equation}
for some positive constant $\,C\,.$
The result is optimal, in the sharp sense.%
\label{laplohas}
\end{theorem}
Note that full regularity
$\,\om_{\la,\,\al}\rightarrow\,\om_{\la,\,\al}\,$ is a little
surprising here. In fact, at the light of theorem \ref{laplolas}, we
could merely expected the intermediate regularity result
$\,\om_{\la,\,\al}\rightarrow\,\om_{\la,\,\al-\,1}\,.$\par%
\begin{proof}
We appeal to the theorem \ref{sufasvero}. Assumptions
\eqref{massim} and \eqref{simplelas} are trivially verified. Let's
prove \eqref{naflitas}.  Set
$$
L(r)=\,\log{\frac1r}\,.
$$
Straightforward calculations show that
\begin{equation}
\oom'(r)=\, r^{\la-\,1}\,L(r)^{-\,\al} \,\big(\,\la
+\,\al\,L(r)^{-\,1}\,\big)%
\label{oprimas}
\end{equation}
and that
\begin{equation}
\oom''(r)=\,-\, r^{\la-\,2}\,L(r)^{-\,\al} \,\Big(\,\la\,(1-\,\la)
-\, (\,2\la -\,1) \,\al \,L(r)^{-\,1}-\,
\al\,(\al+\,1)\,L(r)^{-\,2}\,\Big)\,.%
\label{osegu}
\end{equation}
Equation \eqref{osegu} shows that $\,\oom''(r)<\,0\,$ in a
neighborhood of the origin, since $\, \lim_{r\rightarrow
0}\,L(r)=\,+\,\infty\,.$ Hence $\,\oom\,$ is concave. Furthermore
\eqref{naflitas} holds since
\begin{equation}
\lim_{r\rightarrow 0}\,\frac{\oom(r)}{r\,\,\oom'(r)} =\,\frac1\la >
\,1\,.%
\label{flita222}
\end{equation}%
To prove full regularity we appeal to de l'H\^opital rule and to
\eqref{flita222} to show that
\begin{equation}
\lim_{r\rightarrow 0}\,\frac{\ho(r)}{\,\oom(r)}=\,\lim_{r\rightarrow
0}\,\frac{\oom(r)}{r\,\,\oom'(r)}=\,\frac1\la\,.%
\label{lambes}
\end{equation}%
In particular \eqref{pertinho} holds for $\,r\,$ in some
neighborhood of the origin. Hence proposition \ref{eqnormes} applies.\par%
\end{proof}
It would be interesting to study higher order regularity results in
the framework of H\"olog spaces.
\section{Sharpness of the regularity results.}\label{optimus}%
In this section we prove the \emph{sharpness} of our regularity
results (a simple example was shown at the end of section
\ref{DOPO}). The proof is quite adaptable to different situations,
local and global results, etc. We merely show the main argument. We
construct a counter-example, which concerns constant coefficients
operators (we could easily deny case by case), which shows that any
stronger regularity result can not occur. We start by considering
the Laplace operator $\De$. We remark that the argument applies to
the regularity results stated in theorems \ref{sufasvero} and
\ref{sufasvero-2}. However, in the second theorem, the conclusion is
obvious, due to full regularity.\par%
For convenience, we assume that $\,\oom(r)\,$ is differentiable, and
that there is a positive constant $C$  such that
\begin{equation}
\frac{\oom(r)}{r\,\,\oom'(r)}\geq\,C>\,0\,,%
\label{flotas}
\end{equation}%
for $r>\,0\,,$ in a neighborhood of the origin. Note that \eqref{flotas} holds, with $\,C=\,1\,,$ if $\,\oom(r)$ is concave.\par%
\begin{proposition}
Assume that $\,\oom(r)\,$ satisfies the above hypothesis, and let
$\,\ho_{0}(r)$ be a given oscillation function. Assume that the
results stated in theorem \ref{sufasvero} hold by replacing $\ho$ by
$\ho_{0}\,$. Then there is a constant $c$ for
which $\,\ho(r)\leq\,c\,\ho_0(r)\,.$%
\label{shap}
\end{proposition}
We may say that any regularity result better than \eqref{hahega} is
false.
\begin{proof}
For simplicity, we start by assuming that $\,\bL=\,\De\,.$ Consider
the function
\begin{equation}
u(x)=\,\ho(|x|) \,x_1\,x_2\,,%
\label{emrn}
\end{equation}
defined in $\R^n\,, n\geq\,2\,$. Actually, we are merely interested
in the
behavior near the origin (see \eqref{cotresa-2} below).\par%
Straightforward calculations show that
\begin{equation}
\De\,u(x)=\,(n+\,2)\,\frac{x_1\,x_2}{|\,x\,|^2}\,\oom(x)
+\,\frac{x_1\,x_2}{|\,x\,|^2}\,|\,x\,| \,\oom'(|x|)\,.%
\label{eodel}
\end{equation}%
In particular, $\,\De\,u(0)=\,0\,$. By appealing to \eqref{flotas}
one shows that
$$
|\,\De\,u(x)-\,\De\,u(0)\,|=\, |\,\De\,u(x)\,| \leq\, C\,
\oom(|x|)\,.
$$
Hence, in a neighborhood of the origin, $\,f(x)=\,\De\,u(x)\,$
belongs to $\,D_{\oom}\,.$\par%
On the other hand, straightforward calculations show that
\begin{equation}
\pa_1\,\pa_2\,u(x)=\,\ho(|x|)+\,\frac{1}{|x|^2}\,\big(\,x_1^2+\,x_2^2-\,2\,\frac{x_1^2\,x_2^2}{|x|^2}\,\big)
\cdot\,\oom(|x|)+\,\frac{x_1^2\,x_2^2}{|x|^4}\,\cdot (\,|\,x\,| \,\oom'(|x|)\,)\,.%
\label{papar}
\end{equation}
In particular $\,\pa_1\,\pa_2\,u(0)=\,0\,,$ and
$$
|\,\pa_1\,\pa_2\,u(x)-\,\,\pa_1\,\pa_2\,u(0)\,|\geq\,\ho(|x|)
$$
for $ 0<|x| <<1\,,$ since in equation \eqref{papar} the coefficients
of $\,\oom(|x|)\,$ and of $\,|\,x\,| \,\oom'(|x|)\,$ are
nonnegative. On the other hand, if  $\,\ho_{0}(r)$ regularity holds,
one has
$$
|\,\pa_1\,\pa_2\,u(x)-\,\,\pa_1\,\pa_2\,u(0)\,| \leq\,
(\,c\,\|f\|_\oom \,)\,\ho_{0}(|x|)
$$
for some $c>\,0\,$. Hence $\,\ho(r) \leq\,c_0\,\ho_{0}(r)\,,$ for $
\,r>\,0\,,$ in a neighborhood of the origin.

\vspace{0.2cm}

If $\,\bL\,$  is given by \eqref{elle} we replace \eqref{emrn} by
\begin{equation}
u(x)=\,\ho(|x|) \,\sum_1^n b_{i\,j} x_i\,x_j \,,%
\label{emrn2}
\end{equation}
where $\,B\neq\,0\,$ is symmetric and
$$
 \,\sum_{i,\,j=\,1}^n \, a_{i\,j}\, b_{i\,j}=\,0\,.
$$
In particular, if a specific coefficient $\, a_{k\,l}\,$ vanishes,
we may simply choose $\,u(x)=\,\ho(|x|) \,x_k\,x_l\,,$ as done in
\eqref{emrn}.%

\vspace{0.2cm}

We localize the above result as follows. Assume that $\,0 \in
\Om\,,$ and consider the function
\begin{equation}
u(x)=\,\psi(|x|)\, \ho(|x|) \,x_1\,x_2\,,%
\label{cotresa-2}
\end{equation}
where $\psi(r)\,$ is non-negative, indefinitely differentiable,
vanishes for $\,r\geq\,\rho>\,0\,,$ and is equal to $\,1\,$ for
$\,|x|<\,\frac{\rho}{2}\,.$ The radius $\,\rho\,$ is such that
$\,I(0,\,\rho)\,$ is contained in $\,\Om\,.$ The above truncation
allows us to assume homogeneous boundary conditions in $\,\Om\,$ (we
may consider combinations of functions as above, centered in
different points in
$\,\Om\,$, with distinct radius, and distinct cut-off functions).%
\end{proof}

\vspace{0.2cm}

It is worth noting that in the above argument the specific
expressions of the coefficients of $\,\oom(|x|)\,$ and
$\,|x|\,\oom'(|x|)\,$ are secondary (even if the non-negativity of
these coefficients was exploited). They are homogeneous functions of
degree zero, without particular influence on the minimal regularity.
The crucial point is that the second order derivative
$\,\pa_1\,\pa_2\,u(x)\,,$ due to the term $\,\,x_1\,x_2\,$ in
\eqref{emrn}, leaves unchanged the "bad term" $\,\ho(|x|)\,.$ This
does not occur for derivatives $\,\pa_i^2\,u(x)\,,$ hence does not
occur for $\,\De\,u(x)\,.$\par%
It looks interesting to note that the "bad term" $\,\ho(|x|)\,$ can
not be eliminated by the other two terms which are present in the
right hand side of \eqref{papar}. Even when full regularity occurs
(like in H\"older and H\"olog spaces), the "bad term" $\,\ho(|x|)\,$
is still not eliminated. It simply is as regular as the other two
terms, $\,\oom(|x|)$\, and $\,|\,x\,| \,\oom'(|x|)\,$. See also
\cite{BV-ALBPAO}, section 6, for some comment.
\section{On data spaces larger then $\,\bC_*(\Ov)\,.$ }\label{cestrelas-mais}%
In the context of \cite{BVJDE}, Theorem \ref{laplaces} was
peripheral. Hence, the proof (written in a still existing
manuscript, denoted here [BVUN]), remained unpublished. Actually, at
that time, we proved the above result for more general elliptic
boundary value problems. The proofs depend only on the behavior of
the related Green's functions. Recently, by following the same
ideas, we have shown, see \cite{BVSTOKES}, that for every $\,\ff \in
\,\bC_*(\Ov)\,$ the solution $(\bu,\,p)$ to the Stokes system
\begin{equation}
\left\{
\begin{array}{l}
-\,\De\,\bu+\,\na\,p=\,\ff \quad \textrm{in} \quad \Om \,,\\
\na \cdot\,\bu=\,0  \quad \textrm{in} \quad \Om \,,\\
\bu=\,0 \quad \textrm{on} \quad \Ga%
\end{array}
\right.%
\label{doistokes}
\end{equation}
belongs to $\,\bC^2(\Ov)\times\,C^1(\Ov)\,$.

\vspace{0.2cm}

In the manuscript [BVUN] we also tried to extend the
result claimed in theorem \ref{laplaces} to data belonging to
functional spaces  $\,B_*(\Ov)\,$ \emph{containing} $\,C_*(\Ov)\,.$ By setting
$$
\om_f(x;\,r)= \, \sup_{y \in\,\Om(x;\,r)}\,|\,f(x)-\,f(y)\,|\,,
$$
we may write
\begin{equation}%
[\,f\,]_* =\,\int_0^R  \,\sup_{ \,x \in\,\Ov}\, \om_f(x;\,r)\,
\,\frac{dr}{r}\,.%
\label{catriz}
\end{equation}
So, together with $\,C_*(\Ov)\,,$ we have
considered a functional space $\,B_*(\Ov)\,$ obtained by commuting
\emph{integral} and \emph{sup} operators in the right hand side of
definition \eqref{catriz}: For each $\,f \in\,C(\Ov)\,,$ we defined
the semi-norm
\begin{equation}%
\langle\,f\,\rangle_* = \,\sup_{ \,x \in\,\Ov}\, \int_0^R \,
\om_f(x;\,r)\,
\,\frac{dr}{r}\,,%
\label{seis-bis}
\end{equation}
and a related functional space$\, B_*(\Ov)\,.$ We have shown that
the inclusion $ \,C_*(\Ov) \subset\,B_*(\Ov)\,$ is proper, by
constructing strongly oscillating functions which belong to
$\,B_*(\Ov)\,$ but not to $\,C_*(\Ov)\,.$ This construction was
recently published in reference \cite{BV-LMS}, Proposition 1.7.1.
Furthermore, in [BVUN], we have shown that Theorem \ref{laplaces}
and similar results hold in a weaker form, for data $\,f
\in\,B_*(\Ov)\,$, by proving that the first order derivatives of the
solution $\,u\,$ are Lipschitz continuous in $\,\Ov\,.$  The proof
is published in reference \cite{BV-LMS}, actually for data in a
functional space $\,D_*(\Ov)\,$ containing $\,B_*(\Ov)\,$. See
Theorem 1.3.1 in \cite{BV-LMS}. A similar extension holds for the
Stokes problem, as shown in reference \cite{BVJP}, Theorem 6.1,
where we have proved that if $\,\ff \in \,\bD_*(\Ov)\,,$ then the
solution $\,(\bu,\,p)\,$ of problem \eqref{doistokes} satisfies the
estimate $\, \|\,\bu\,\|_{1,\,1}+\,\|\,p\,\|_{0,\,1} \leq\,C\,
|\|\,\ff\,|\|_*\,.$ Full regularity for data in $\,\bB_*(\Ov)\,$
would follow from a possible density of regular functions in this
last space, a challenging open problem. The simple proof would be
obtained by replacing the space $ \,\bC_*(\Ov)\,$ by
$\,\bB_*(\Ov)\,$ in \cite{BVSTOKES}, section 4. A similar remark
holds for $\,\bD_*(\Ov)\,.$ However, in this last case, the desired
density result looks quite unlikely.

\end{document}